\documentclass[paper=a4, fontsize=12pt]{scrartcl}	

\usepackage[english]{babel}															
\usepackage[protrusion=true,expansion=true]{microtype}				
\usepackage{amsmath,amsfonts,amsthm}										
\usepackage[pdftex]{graphicx}														
\usepackage[hang, small,labelfont=bf,up,textfont=it,up]{caption}	
\usepackage{epstopdf}																	
\usepackage{subfig}																		
\usepackage{booktabs}																	

\usepackage{verbatim}
\usepackage{fancyvrb}
\DefineShortVerb{\|}								

\usepackage{sectsty}								
\allsectionsfont{
	\usefont{OT1}{bch}{b}{n}
	}

\sectionfont{
	\usefont{OT1}{bch}{b}{n}
	\sectionrule{0pt}{0pt}{-5pt}{0.8pt}
	}

\usepackage{fancyhdr}
\numberwithin{equation}{section}															
\numberwithin{figure}{section}																
\numberwithin{table}{section}																

\newcommand{\R}{\mathbb R}

\newcommand{\rank}{\mbox{rank}}

\newtheorem{theorem}{Theorem}  
\newtheorem{corollary}[theorem]{Corollary}
\newtheorem{lemma}[theorem]{Lemma}
\newtheorem{proposition}[theorem]{Proposition}

\usepackage{amsbsy}
\newcommand{\e}{\pmb{e}}
\newcommand{\ld}{\pmb{\lambda}}
\newcommand{\x}{\pmb{x}}

\usepackage{color}

\title{Equilibrium Locus of The Flow on Circular Networks of Cells}
\author{Yirmeyahu J. Kaminski\\
Depart. of Applied Mathematics, \\
Holon Institute of Technology, \\
Holon, Israel.}

\date{}

\begin{document}

\maketitle

\begin{abstract}
We perform a geometric study of the equilibrium locus of the flow that models the diffusion process over a circular network of cells. We prove that when considering the set of all possible values of the parameters, the equilibrium locus is a smooth manifold with corners, while for a given value of the parameters, it is an embedded smooth and connected curve. For different values of the parameters, the curves are all isomorphic. 

Moreover, we show how to build a homotopy between different curves obtained for different values of the parameter set. This procedure allows the efficient computation of the equilibrium point for each value of some first integral of the system. This point would have been otherwise difficult to be computed for higher dimensions. We illustrate this construction by some numerical experiments.

Eventually, we show that when considering the parameters as inputs, one can easily bring the system asymptotically to any equilibrium point in the reachable set, which we also easily characterize. 
\end{abstract}

\section{Introduction}

A circular network of cells is a set of cells that are connected in along a ring such that the last cell is connected to the first. An instance of this kind of model is the ribosome flow model on a ring has been introduced in~\cite{Raveh-all-2015}. 

Such a model is a parametric dynamical system defined as follows:
$$
\left \{ \begin{array}{ccl}
\dot{x}_1 & = & \lambda_n x_n(1-x_1) - \lambda_1 x_1(1-x_2) \\
\dot{x}_2 & = & \lambda_1 x_1(1-x_2) - \lambda_2 x_2(1-x_3) \\
\vdots & & \\
\dot{x}_{n-1} & = & \lambda_{n-2} x_{n-2}(1-x_{n-1}) - \lambda_{n-1} x_{n-1}(1-x_n) \\
\dot{x}_n & = & \lambda_{n-1} x_{n-1}(1-x_n) - \lambda_n x_n(1-x_1) 
\end{array} \right .
$$

The parameters $\lambda_1, \cdots, \lambda_n$ are real strictly positive numbers. They define the degree of diffusion between the cells. As shown in~\cite{Raveh-all-2015}, if the initial point lies in the hypercube $[0,1]^n$, then the system always stays within these limits. Then this system models the occupancy levels of a circular chain of $n$ sites (for example a circular DNA), while $\lambda_1, \cdots, \lambda_n$ are transition rates.   

In this work, we focus on the study of the equilibrium locus of this system and in particular we study its dependency on the parameters.

We prove that the set of equilibrium points $(\lambda_1, \cdots, \lambda_n, e_1, \cdots, e_n)$ of the system, denoted $V_n$, is a smooth sub-manifold with corners of $R^{2n}$ which dimension is $n+1$. 

We also consider the projection onto the first summand $\pi: V_n \rightarrow (R_+^*)^n$,\\ $(\lambda_1, \cdots, \lambda_n, e_1, \cdots, e_n) \mapsto (\lambda_1, \cdots, \lambda_n)$ and prove it is a proper surjective submersion. Relying on Thom's First Isotopy theorem~\cite{Mather-2012}, we prove that for a given value of the parameters, the set of equilibrium points is an embedded smooth curve. Also since the base space of the projection is connected, for different values of the parameters, the curves are all isomorphic. Considering the fiber over $(1, \cdots, 1) \in (R_+^*)^n$, we show that all these curves at connected. 

Moreover, we show how to build a homotopy between different curves obtained for different values of the parameter set. This procedure allows the efficient computation of the equilibrium point for each value of some first integral of the system. This point would have been otherwise difficult to be computed for higher dimensions. We illustrate this construction by some numerical experiments.   

This homotopy path tracing method is proven to be valid, as we prove the existence of the path the numerical method allows to trace. 

Eventually, we show that when considering the parameters as inputs, one can easily bring the system asymptotically to any equilibrium point in the reachable set, which we also easily characterize.

\section{Parametric Manifold}
\label{sec::parametric_manifold}

Given the parameters $\lambda_1, \cdots, \lambda_n \in (\R_{+}^*)^n$, the equilibrium point $(e_1, \cdots, e_n) \in [0,1]^n$ must satisfy $f(\lambda_1, \cdots, \lambda_n, e_1, \cdots, e_n) = 0$, where:
$$
f(\ld,\e) = \left( \begin{array}{c} 
\lambda_n e_n (1-e_1) - \lambda_1 e_1 (1-e_2) \\
\vdots \\
\lambda_n e_n (1-e_1) - \lambda_{n-1} e_{n-1} (1-e_n)
\end{array} \right). 
$$

Here it is worth noting that $f$ has values in $\R^{n-1}$ and not in $\R^n$, since the vanishing of the $n-1$ first components of the vector field implies immediately the vanishing of the last one, which is up to the sign merely the sum of the $n-1$ first components. In addition the vanishing of all the $n-1$ first components of the vector field is easily seen to be equivalent to the vanishing of the $n-1$ components of the function $f$ defined above. 

For the sake of simplicity, we have denoted $\ld$ for $\lambda_1, \cdots, \lambda_n$ and $\e$ for $e_1, \cdots, e_n$. We shall now prove that $V_n = f^{-1}(0)$ is a smooth manifold with corners. 

The Jacobian matrix of $f$ with respect to $\ld$ is given by:
$$
J_{\ld} = \left[ \begin{array}{ccccc}
-e_1(1-e_2) & 0 & \cdots & 0 & e_n(1-e_1) \\
0 & -e_2(1-e_3) & \cdots & 0 & e_n(1-e_1) \\
\vdots & \vdots & \vdots & \vdots & \vdots \\
0 & \hdots & \hdots & -e_{n-1}(1-e_n) & e_n(1-e_1)
\end{array} \right]
$$

We want to determine the rank of this matrix. We shall use the following lemma.

\begin{lemma}
\label{lemma::boundary}
If $e_i = 0$ for some $i$, then $e_1 = \cdots = e_n =0$. If $e_i = 1$ for some $i$, then $e_1 = \cdots = e_n = 1$. 
\end{lemma}
%
%

This lemma is algebraically obvious and has a clear physical meaning. If the system is saturated in some bin, it is the case as well in all bins. If one bin is empty, so must be the other bins. 

Furthermore this lemma has an important mathematical implication. More explicitly this says that the only points $(\ld,\e)$ of $V_n$ for which $\e$ lies on the border of the hypercube $[0,1]^n$ are precisely $\e = (0,\cdots,0)$ and $\e = (1,\cdots,1)$. Thus when investigating the dependency on $\e$ of any claim concerning $V_n$, it is will useful to look separately at the three following cases: $\e \in (0,1)^n, \e = (0,\cdots,0)$ and $\e = (1,\cdots,1)$.

For $\e \in (0,1)^n$, the $J_{\ld}$ has full rank, that is $n-1$. For $\e = (0,\cdots,0)$ or $\e = (1,\cdots,1)$, $J_{\ld} = 0$. 

On the other hand, the Jacobian matrix of $f$ with respect to $\e$ is given by:

{\small
$$
J_{\e} = \left[ \begin{array}{cccccc} 

-\lambda_n e_n -\lambda_1(1-e_2) & \lambda_1 e_1 & 0 & \hdots & 0 & \lambda_n(1-e_1) \\
-\lambda_n e_n & -\lambda_2(1-e_3) & \lambda_2 e_2 & \hdots & 0 & \lambda_n(1-e_1) \\
\vdots & \vdots & \vdots & \vdots & \vdots & \vdots \\
-\lambda_n e_n & 0 & \hdots & 0 & -\lambda_{n-1}(1-e_n) & \lambda_n(1-e_1) + \lambda_{n-1} e_{n-1}
\end{array} \right]
$$
}

If $\e = (0,\cdots,0)$, then $J_{\e}$ is in row echelon form and no pivot vanishes. Thus in that case $J_{\e}$ has full rank. If $\e = (1,\cdots,1)$, a short computation also shows the matrix has full rank. 

From this simple argument we conclude that the level set $V_n$ is indeed a properly embedded sub-manifold~\cite{Lee-2013} of $(\R_{+}^*)^n \times [0,1]^n$ with corners such that $\partial V_n = (\R_{+}^*)^n \times \{(0, \cdots, 0), (1, \cdots, 1)\} = (\R_{+}^*)^n \times \{(0, \cdots, 0)\} \cup (\R_{+}^*)^n \times \{(1, \cdots, 1)\}$. Moreover $\dim(V_n) = 2n-(n-1) = n+1$. 

\section{The Fibers}

While our primary goal is to study the fiber over $\ld$ for each possible of value of the parameters, we start our investigation by considering the fiber over a given value of $\e$. This will allow in the sequel a deeper understanding of the main subject.

\subsection{Fibers Over $\e$}
\label{sec::fibers_over_e}

Given a point $\e \in [0,1]^n$, the fiber over $\e$ is by definition the set $\sigma_{\e} = \{\ld \in (\R_{+}^*)^n : f(\ld,\e) = 0\}$. 

When either $\e = (0, \cdots, 0)$ or $\e = (1, \cdots, 1)$, the set $\sigma_{\e}$ is obviously $(\R_{+}^*)^n$. 

Following lemma~\ref{lemma::boundary}, we consider now a given value $\e = (e_1, \cdots, e_n) \in (0,1)^n$. In that case the set of $\sigma_{\e}$ is simply of the form $\R_{+}^* \omega_{\e}$ with $\omega_{\e} \in (\R_{+}^*)^n$, since the system defining $\sigma_{\e}$ is linear in $\ld$ and homogeneous and the rank of $J_{\ld}$ is $n-1$. In the sequel, we shall assume that $\omega_{\e}$ is normalized: $\Vert \omega_{\e} \Vert = 1$, so that it is uniquely defined. Given the geometric form of $\sigma_{\e}$, we can conclude for two different points $\e$ and $\e'$ both in $[0,1]^n$, the fibers $\sigma_{\e}$ and $\sigma_{\e'}$ either coincide or has empty intersection. In the latter case, their closures intersect at the origin.

\subsection{Fibers Over $\ld$}
\label{sec::fiber_over_lambda}

Now let us consider the canonical projection: $\pi: V_n \longrightarrow (\R_{+}^*)^n, (\ld,\e) \mapsto \ld = (\lambda_1, \cdots, \lambda_n)$.  

Let us we denote $\tilde{\pi}$ the projection from $\R^n \times \R^n$ onto the first summand $\R^n$. Then $\pi$ is the restriction of $\tilde{\pi}$ to $V_n$. Also the differential of $\pi$ at some point $p$ of $V_n$ is the restriction of $d\tilde{\pi}$ to the tangent space $T_p(V_n)$ of $V_n$ at $p$. The projection $\pi$ fails to be a local submersion at points $p$ where $\dim(\tilde{\pi}(T_p(V_n))) < n$ or equivalently if $\dim(T_p(V_n)) \cap \ker(d\tilde{\pi}_p) > 1$. This would happen if the rank of the Jacobian matrix of $f$ with respect to $\e$ at $p$ is strictly less than $n-1$. Indeed we have the following sequence of facts:
\begin{enumerate}
\item $T_p(V_n) = \ker(df_p)$,
\item $T_p(V_n) \cap \ker(d\tilde{\pi}_p) = \{(0,\cdots,0,u_1,\cdots,u_n) \mid df_p(0,\cdots,0,u_1,\cdots,u_n) = 0\}$,
\item $\dim(T_p(V_n) \cap \ker(d\tilde{\pi}_p)) > 1 \Leftrightarrow \rank(J_{\e}(p)) < n-1$.
\end{enumerate}

Now let us show that $\rank(J_{\e}(p)) < n-1$ never occurs. Recall that:

{\small
$$
J_{\e} = \left[ \begin{array}{cccccc} 

-\lambda_n e_n -\lambda_1(1-e_2) & \lambda_1 e_1 & 0 & \hdots & 0 & \lambda_n(1-e_1) \\
-\lambda_n e_n & -\lambda_2(1-e_3) & \lambda_2 e_2 & \hdots & 0 & \lambda_n(1-e_1) \\
\vdots & \vdots & \vdots & \vdots & \vdots & \vdots \\
-\lambda_n e_n & 0 & \hdots & 0 & -\lambda_{n-1}(1-e_n) & \lambda_n(1-e_1) + \lambda_{n-1} e_{n-1}
\end{array} \right]
$$  
}

Let $A$ the matrix obtained from $J_{\e}$ by discarding the first column. Then $A$ is a square $(n-1) \times (n-1)$ matrix. Let us consider the expansion of its determinant with respect to the first line:
\begin{align*}
\lambda_1 e_1 [\lambda_2 e_2 \cdots \lambda_{n-2} e_{n-2} (\lambda_n (1-e_1) + \lambda_{n-1} e_{n-1})] + \\
(-1)^n \lambda_n (1-e_1)[(-\lambda_2(1-e_3)) \cdots (-\lambda_{n-1}(1-e_n))] = \\
\lambda_1 e_1 [\lambda_2 e_2 \cdots \lambda_{n-2} e_{n-2} (\lambda_n (1-e_1) + \lambda_{n-1} e_{n-1})] + \\
\lambda_n (1-e_1)[\lambda_2(1-e_3) \cdots \lambda_{n-1}(1-e_n)] 
\end{align*}


When the coordinates of $\e$ neither vanish nor are equal to $1$ (see lemma~\ref{lemma::boundary}), it is a sum of strictly positive terms. Thus it never vanishes. When $\e = (0,\cdots,0)$ or $\e = (1,\cdots,1)$, as mentioned in section~\ref{sec::parametric_manifold}, the matrix $J_{\e}$ has also full rank. This latter fact can also be seen directly from the explicit expression of the determinant.  

Eventually, these simple considerations yield the following lemma:

\begin{lemma}
The canonical projection $\pi: V_n \longrightarrow (\R_{+}^*)^n$ is a submersion at every point of $V_n$. 
\end{lemma}

Note that one could reach the same conclusion by considering the discriminant variety of the system~\cite{Lazard-Rouillier-2007}.

Now we turn our attention to the surjectivity of $\pi$ and prove the following lemma:

\begin{lemma}
The canonical projection $\pi: V_n \longrightarrow (\R_{+}^*)^n$ is surjective. 
\end{lemma}
\begin{proof}
Consider $\ld \in (\R_{+}^*)^n$. Let $m < \min(\lambda_1, \cdots , \lambda_{n-1})$. 

Consider $\tau \in (0,1)$. Then we can choose $\alpha \in (0,1)$, such that $\lambda_n \alpha < m \tau$. Similarly, we can choose $\beta_1 , \cdots, \beta_{n-1} \in (0,1)$, such that $\lambda_n \alpha = \lambda_i \beta_i$ for all $i = 1, \cdots , n-1$. 

We shall now prove that we can find $(e_1, \cdots, e_n) \in (0,1)^n$ such that $\alpha = e_n(1-e_1)$ and $\beta_i = e_i(e-e_{i+1})$ for all $i=1, \cdots, n-1$, provided we choose $\tau$ wisely. 

For that purpose, consider the map: $\phi: [0,1]^n \longrightarrow [0,1]^n, (e_1, \cdots, e_n) \mapsto (e_n(1-e_1), e_1(1-e_2), \cdots, e_{n-1}(1-e_n))$. The map is continuous and defined on a compact set. Consider the subset $H = \{\e \in [0,1]^n : \sum_i e_i = 1\}$. Being a closed subset of compact set, it is compact as well. Therefore $\phi(H)$ is compact. Now let $N$ or $\Vert . \Vert$ be the Euclidean norm in $\R^n$ and let $\psi = N \circ \phi$. Let $\rho$ be the minimum of $\{\Vert x \Vert : x \in \phi(H)\} = \psi(H)$. Of course $\rho \in \psi(H)$. Let ${\bf u} = (u_1, \cdots, u_n) \in [0,1]^n$, such that $\psi(u) = \rho$. If $\rho = 0$, then $u_n(1-u_1) = u_1(1-u_2) = \cdots = u_{n-1}(1-u_n) = 0$. This implies that either ${\bf u} = (0, \cdots ,0)$ or ${\bf u} = (1, \cdots, 1)$. In both cases, ${\bf u} \not \in H$. Therefore we have $\rho > 0$. 

In addition if one denotes by $(a_1 , \cdots, a_n)$ the standard basis of $\R^n$, each segment $s_i = \{\eta a_i : 0 \leq \eta \leq 1\}$ is mapped by $\phi$ to $s_{i+1}$. In order to avoid cumbersome notations, we use the following rule: if $(k+1)n \geq j > kn$, $s_j$ (respectively $a_j$) denotes in fact $s_{j-kn}$ (respectively $a_{j-kn}$). With this rule granted, let us define the triangles: $f_i = \{t \eta a_i + (1-t) \eta a_{i+1} : \eta, t \in [0,1]\}$.    

Let $T \subset \R$ be the $n-$dimensional pyramid which basis is $H$ and which top vertex is $0$. The boundary $\partial T$ of $T$ is the union $\partial T = f_1 \cup \cdots \cup f_n \cup H$. The image of $T$ by $\phi$ is the connected set whose boundary is $\phi(f_1) \cup \cdots \cup \phi(f_n)  \cup \phi(H)$. Note that $\phi(f_i) = \{(t \eta - t(1-t)\eta^2) a_{i+1} + (1-t) \eta a_{i+2} : \eta, t \in [0,1]\} \subset f_{i+1}$.

Therefore by the above considerations, $\phi(T)$ contains the upper-right quadrant of the ball centered at the origin and of radius $\rho$, i.e $B(0,\rho) \cap T$. 

It is now sufficient to choose $\tau$ such that $m \tau < \rho$. Then $(\alpha, \beta_1, \cdots, \beta_{n-1}) \in B(0,\rho) \cap \phi(T)$. Then for the given values of $\lambda_1, \cdots, \lambda_n$, there exists $\e \in (0,1)^n$, such that $\pi(\ld,\e) = \ld$. 
 \end{proof}

Finally we shall prove this last lemma:

\begin{lemma}
The canonical projection $\pi$ is proper.
\end{lemma}
\begin{proof}
Let $K$ be a compact set in $(\R^{*}_{+})^n$. We need to prove that $\pi^{-1}(K)$ is compact in the subspace topology of $V_n$.
Since $\pi^{-1}(K) = \{(\ld,\e) \in K \times [0,1]^n : f(\ld,\e) =0\} = (K \times [0,1]^n) \cap V_n$, it a closed subset of a compact set. Therefore it is compact as well in $(K \times [0,1]^n)$. Thus $\pi^{-1}(K)$ is also compact in $V_n$.   
\end{proof}

Stacking together these three results, we get the following theorem:
\begin{theorem}
\label{theo::fibers}
For each $\ld \in (\R^{*}_{+})^n$, the fiber of the canonical projection $\pi$ is a one-dimensional smooth manifold. Moreover all the fibers are isomorphic.
\end{theorem}
\begin{proof}
Since the map $\pi$ is a surjective submersion, all the fibers are smooth manifold of dimension $\dim(V_n) - n = n+1 -n = 1$. In addition, since the map is proper, by Thom's First Isotopy Theorem~\cite{Mather-2012}, the map is locally a trivial fibration. Since $(\R^{*}_{+})^n$ is connected, the fibers are all isomorphic one each other.    
\end{proof}

One can wonder if two fibers can coincide or intersect elsewhere than at $(0, \cdots, 0)$ and $(1, \cdots, 1)$. Consider two given points $\ld$ and $\ld'$ in $(\R_{+}^*)^n$. If there exists $\e \in (0,1)^n$, which lies on both fibers obver $\ld$ and $\ld'$, then relying on the notations of section~\ref{sec::fibers_over_e}, we observe that $\ld, \ld' \in \sigma_{\e}$. Therefore one can conclude with the following lemma.

\begin{lemma}
For $\ld$ and $\ld'$ in $(\R_{+}^*)^n$, the fibers over $\ld$ and $\ld'$ coincide if and only if $\ld$ and $\ld'$ are collinear, . 
\end{lemma}

Now we shall prove that all fibers are connected. Given theorem~\ref{theo::fibers}, it is sufficient to prove that the fiber over $\ld = \pmb{1} = (1, \cdots, 1)$ is connected. As a matter of fact, in the sequel, the vector $\alpha \pmb{1}$ for $\alpha \in \R$ will be denoted in bold font as follows $\pmb{\alpha}$. 

For $\ld \in (\R^{*}_{+})^n$, let us simply write $\pi_{\ld} = \{\e \in [0,1]^n : (\ld,\e) \in V_n,\pi(\ld,\e) = \ld \}$ for the fiber. It is obvious that $\pi_{\pmb{1}}$ contains the segment $[\pmb{0},\pmb{1}]$. We want to prove that $\pi_{\pmb{1}}$ is restricted exactly to this segment. For this purpose we shall consider the projection $\phi: \pi_{\pmb{1}} \backslash \{\pmb{0},\pmb{1}\} \longrightarrow (0,1), \e \mapsto e_1$.

\begin{lemma}
The map $\phi$ is surjective.
\end{lemma}
\begin{proof}
Indeed for any value $\alpha \in (0,1)$, the point $\pmb{\alpha}$ lies in $\pi_{\pmb{1}} \backslash \{\pmb{0},\pmb{1}\}$ and $\phi(\pmb{\alpha}) = \alpha$. 
\end{proof}

\begin{lemma}
The map $\phi$ is a submersion.
\end{lemma}
\begin{proof}
The map $\phi$ fails to be a submersion at points where the kernel of the Jacobian matrix $J_{\e}$ (for $\ld = \pmb{1}$) is orthogonal to the vector $a_1 = [1,0,\cdots,0]^t$. 

Since the kernel is a one-dimension real vector space, for this condition to be satisfied, it is necessary that the $(n-1)-$last columns of the Jacobian matrix $J_{\e}$ are linearly dependent. However, as mentioned above the determinant of this sub-matrix is a sum of strictly positive terms. Thus it never vanishes on $\pi_{\pmb{1}} \backslash \{\pmb{0},\pmb{1}\}$. Hence, the map $\phi$ is indeed a submersion over $\pi_{\pmb{1}} \backslash \{\pmb{0},\pmb{1}\}$. 
\end{proof}

\begin{lemma}
The map $\phi$ is proper. 
\end{lemma}
\begin{proof}
Consider a compact subset $K \subset (0,1)$. Then $K$ is included in some closed interval $K \subset [a,b] \subset (0,1)$. Therefore $\phi^{-1}(K)$ is a closed subset of $\pi_{\pmb{1}} \backslash \{\pmb{0},\pmb{1}\}$, also included in $\phi^{-1}([a,b])$. The latter set is closed and bounded in $\R^n$ and thus is compact. Therefore $\phi^{-1}(K)$ is a closed subset of a compact and is compact as well in $\R^n$. Since $\pi_{\pmb{1}}$ is closed in $\R^n$, $\phi^{-1}(K)$ is also compact in $\pi_{\pmb{1}}$. Finally, since $\phi^{-1}(K) \subset \pi_{\pmb{1}} \backslash \{\pmb{0},\pmb{1}\}$, it is compact in the subset topology of the $\pi_{\pmb{1}} \backslash \{\pmb{0},\pmb{1}\}$.     
\end{proof}

\begin{corollary}
The fibers of $\phi$ are all isomorphic.
\end{corollary}
\begin{proof}  
Given the threee previous lemmas and by Ehresmann's theorem~\cite{Ehresmann}, $\phi$ is a locally trivial fibration \footnote{In this case, we made use of Ehresmann's theorem, while in theorem~\ref{theo::fibers}, we used the Thom first isotopy theorem, which is a generalization of it. In the latter case the use of the more general setting was necessary since $V_n$ is manifold with corners, while here $\pi_{\pmb{1}} \backslash \{\pmb{0},\pmb{1}\}$ is a mere manifold and the classical setting of Ehresmann is enough.}. Since $(0,1)$ is connected, all the fibers of $\phi$ are isomorphic. 
\end{proof}
 
Therefore, it is sufficient to compute the fiber over $\frac{1}{2}$ to understand the common structure of all the fibers. 

So consider that $e_1 = \frac{1}{2}$. Assume that $e_n > \frac{1}{2}$. Then given that $e_{n-1} = \frac{e_n}{2(1-e_n)}$, we have $e_{n-1} > e_n$. On the one hand, by induction we get a decreasing sequence $e_2 > e_3 > \cdots > e_n$. On the other hand, we have $e_2 = 1-e_n > e_n$. Thus $e_n < \frac{1}{2}$. Thus we ends up with a contradiction. Now let us assume that $e_n < \frac{1}{2}$. Then we get an increasing sequence $e_2 < e_3 < \cdots < e_n$. Together with $e_2 = 1-e_n$, we get this time that $e_n > \frac{1}{2}$. All together, we get that $e_n = e_1 = \frac{1}{2}$. Then it is obvious that $e_1 = \cdots = e_n = \frac{1}{2}$. Hence the fiber is made of a single point. Therefore all the fibers are singletons.
Thus there is a bijective submersion from $\pi_{\pmb{1}} \backslash \{\pmb{0},\pmb{1}\}$ to $(0,1)$. There the two manifolds are diffeomorphic. 
Then $\pi_{\pmb{1}} \backslash \{\pmb{0},\pmb{1}\}$ is connected. Stacking this to theorem~\ref{theo::fibers}, this yields the following corollary:

\begin{corollary} 
For each value of $\pmb{\lambda} \in (\R_{+}^{*})^n$, the fiber $\pi_{\ld}$ is a smooth connected curve. Moreover the fiber $\pi_{\pmb{1}}$ is the segment $[\pmb{0},\pmb{1}]$. 
\end{corollary}

\section{Hyperplane Section of A Fiber and Homotopy Continuation}

We are now interested in analyzing the intersection between the fiber $\pi_{\ld}$ and the hyperplane $H_s = \{\e \in (0,1)^n : \sum_i e_i = s\}$. Obviously, this requires that $s \in [0,n]$. As mentioned in~\cite{Raveh-all-2015}, for each $s$ and each $\pmb{\lambda}$, there is a single equilibrium point in $H_s$. Moreover starting from any point in $H_s$, the system converges toward this equilibrium point.   
 
Here we shall prove that indeed the intersection $V_{s,n} = H_s \cap V_n$ is a smooth manifold isomorphic to $(\R^n_+)^n$. More precisely, we have the following proposition. 

\begin{proposition}
\label{prop::covering_map}
For $s \in (0,n)$, consider the intersection $V_{s,n} = H_s \cap V_n$. Then $V_{s,n}$ is a $n-$dimensional smooth variety and the restriction of the projection $\pi$ to $V_{s,n}$ is a diffeomorphism. 
\end{proposition}
\begin{proof}
First we shall prove that $V_{s,n}$ is indeed a smooth manifold. The set $V_{s,n}$ is defined as $f^{-1}(0) \cap h_s^{-1}(0)$ for $h_s(\ld,\e) = \sum_i e_i - s$. For any $s \in (0,n)$ and any $n \geq 1$, $V_{s,n} \neq \emptyset$. Indeed for a given value of $\ld$, the fiber $\sigma_{\ld}$ is a connected curve that joins the two extreme points $(0,\cdots,0)$ and $(1,\cdots,1)$, which are located each in each half space that $H_s \cap \R^n$ defines. 

Consider the function $g(\ld,\e) = (f(\ld,\e),h_s(\ld,\e)) \in \R^n$ and its Jacobian matrix at a point $(\ld,\e) \in V_{s,n}$:
{\small
$$
W = \left[ \begin{array}{cccccc} 

-\lambda_n e_n -\lambda_1(1-e_2) & \lambda_1 e_1 & 0 & \hdots & 0 & \lambda_n(1-e_1) \\
-\lambda_n e_n & -\lambda_2(1-e_3) & \lambda_2 e_2 & \hdots & 0 & \lambda_n(1-e_1) \\
\vdots & \vdots & \vdots & \vdots & \vdots & \vdots \\
-\lambda_n e_n & 0 & \hdots & 0 & -\lambda_{n-1}(1-e_n) & \lambda_n(1-e_1) + \lambda_{n-1} e_{n-1} \\
1 & 1 & \cdots & 1 & 1 & 1
\end{array} \right]
$$  
}

Above  in section~\ref{sec::fiber_over_lambda},  we have introduced the matrix $A$ which is the matrix obtained by discarding the first column and the last row of $W$. We have shown that $A$ is invertible for any value of $(\ld,\e) \in V_n$. Let $v$ be the first column of $W$ without the last element. Since $\rank(A) = n-1$, $v$ is a linear combination of the columns of $A$. Moreover $v$ is collinear with no column of $A$. Therefore $\det(A - V) \neq 0$ where $V$ is the $(n-1) \times (n-1)$ matrix which columns are all identically $v$. As easily seen $\det(W) = \det(A - V) \neq 0$. Thus $W$ is invertible for all $(\ld,\e) \in V_{s,n}$. Therefore $V_{s,n}$ is indeed a smooth variety of dimension $n$.

Now consider the restriction of the projection $\pi: V_n \rightarrow (\R^*_+)^n$ to $V_{s,n}$. It is a proper as $\pi$ is. It is surjective, since as argued just above each fiber of $\pi$ meets $H_s$ and $\pi$ is surjective. It is a submersion because $\pi$ is a submersion and the kernel of $d\pi$ at $(\ld,\e)$ is precisely the tangent space to the fiber and $V_{s,n}$ is transverse to each fiber in $V_n$. Indeed the tangent space to $V_{s,n}$ is precisely the direction of $H_s$, which is transverse to $V_n$ in $\R^{2n}$, since $\rank(W) = n$. 

Eventually, we can conclude that $\pi_{\vert V_{s,n}}$ is a proper surjective submersion. Again by Erhesmann's theorem, its fibers are all isomorphic since $(R_+^*)^n$ is connected. But over $(1,\cdots,1)$ the fiber is a singleton, since the fiber of $\pi$ is a segment that intersects the hyperplane $H_s$ at a single point. As a consequence $\pi_{\vert V_{s,n}}$ is a bijective submersion and is therefore a diffeomorphism.
\end{proof}

We shall rely on this proposition to prove the existence of the construction we introduce now.
  
In the simple case of $\pi_{\pmb{1}}$, we have immediately $H_s \cap \pi_{\pmb{1}} = \{\pmb{\frac{s}{n}}\}$. Here (as in the sequel) we have used the rule introduced in section~\ref{sec::fiber_over_lambda}, so that $\pmb{\frac{s}{n}} = (\frac{s}{n}, \cdots, \frac{s}{n}) \in \R^n$.  In this section, we shall show how to find the equilibrium point $\pi_{\ld} \cap H_s$ for any admissible value of the parameters $\ld$. 

Consider a straight segment line in $(\R_{+}^{*})^n$ joining $\ld_1 = \pmb{1}$ to another point $\ld_2$: $\ld(t) = (1-t)\ld_1 + t\ld_2$ for $0 \leq t \leq t$. This yields an homotopy between the fiber $\pi_{\ld_1} \cap H_s$ and the fiber $\pi_{\ld_2} \cap H_s$:
$$
H(\e,t) = (1-t)g(\ld_1,\e) + t g(\ld_2,\e),
$$ 
where $g(\ld,\e) = [f(\ld,\e),\sum_i e_i - s]^t \in \R^{n+1}$.
 
As we have seen solving $g(\ld_1,\e) = 0$ is extremely easy. For a given $s \in (0,n)$, starting from the solution $\{\pmb{\frac{s}{n}}\}$, we can follow a path embedded in $H_s$ that ends at the solution of $g(\ld_2, \e) = 0$. Consider $\gamma(t) = [x_1(t), \cdots, x_n(t)]^t$ such a path. By construction, we want it to satisfy for all $t \in [0,1]$: $H(\gamma(t),t) = 0$. This yields the following differential system:
$$
\sum_{j=1}^n \frac{\partial H_i}{\partial x_j}(\gamma(t),t) \frac{d x_j}{d t}(t) + \frac{\partial H_i}{\partial t}(\gamma(t),t) = 0, 
$$
where $H_1, \cdots, H_{n+1}$ are the components of $H$. This is a special case of homotopy continuation methods~\cite{Li-97}. This system yields a Newton like method for tracing the path $\gamma(t)$, since we have:
$$
\frac{d \gamma}{d t} (t) = - \left (J(H)_{\gamma(t)} \right)^{-1} \frac{\partial H}{\partial t}(\gamma(t),t),
$$
where $J(H)_{\gamma(t)}$ is the Jacobian matrix of $H$ computed for the $n$ first coordinate at the point $\gamma(t)$. One can easily write this differential equation explicitly.  

The existence of a solution to this non-linear equation on $[0,1]$ is granted as follows.

The function $g$ is linear with respect to $\lambda$. Thus $H(\gamma(t),t) = g((1-t)\ld_1 + t \ld_2,\gamma(t)) = 0$. Therefore $\gamma(t)$ is always an equilibrium point for each value of $t$. Hence the curve $t \mapsto ((1-t)\ld_1+t\ld_2,\gamma(t))$ is a lift of the segment joining $\ld_1$ to $\ld_2$ on the variety $V_{s,n}$. Such a lift exists and is unique by proposition~\ref{prop::covering_map}. 

In our case, the system is particularly simple since $H$ is a quadratic function of $\gamma(t)$. The procedure can be easily implemented. We have implemented the path tracing procedure in Python. In the figure~\ref{fig::homotopy_3d}, we show the homotopy path generated for a system with 3 bins. 

\begin{figure}[h!]
\begin{center}
\includegraphics[scale=0.3]{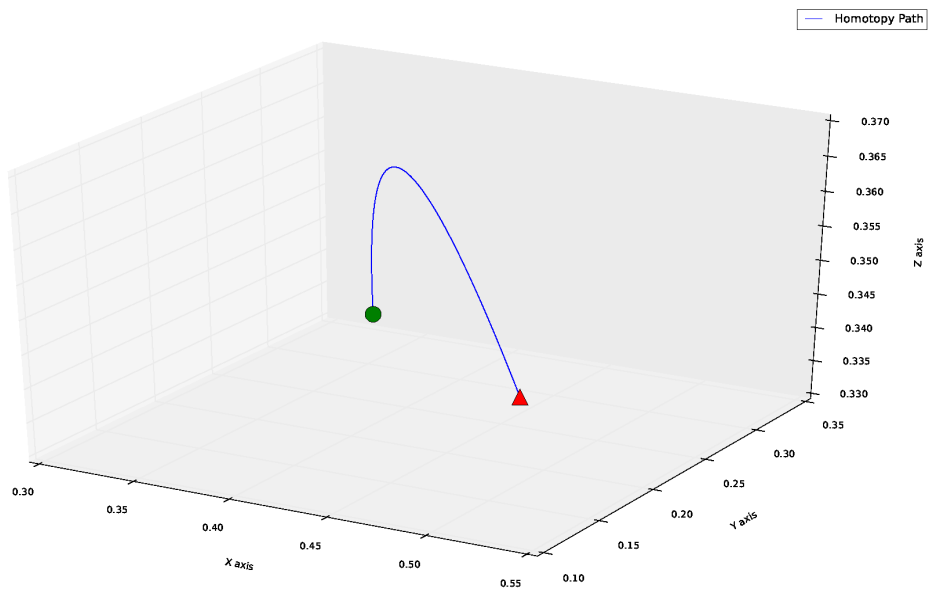}
\caption{The homotopy path generated with the following parameters $\ld_1 = (1,1,1)$, $\ld_2 = ( 1.39328599, 8.30098374, 3.98355604)$ and $s = 1$. The first equilibrium point $E_0 = (1/3,1/3,1/3)$ is represented by a green ball, while the final equilibrium point $E_1 = (  0.53112814, 0.1203633, 0.34850856)$ is rendered as a red triangle. Of course all over the path the constraint $e_1 + e_2 + e_3 = s$ holds.} \label{fig::homotopy_3d}
\end{center}
\end{figure}

We have also run the homotopy path tracing procedure with $n = 100$. The procedure allows to find efficiently the right equilibrium point in the second fiber, which would have been otherwise difficult to determine.

\section{Control}

As shown in section~\ref{sec::fibers_over_e}, for a given value of a desirable equilibrium $\e=(e_1, \cdots, e_n)$, one can find easily the values parameter set $\ld$ that would lead the system to this point, provided the initial state $\x(0) = (x_1(0), \cdots, x_n(0))$ satisfies $\sum_i^n x_i(0) = \sum_i^n e_i$. 

Hence the reachable set for a given initial state vector $x(0)$ is obviously the intersection between the cube $[0,1]^n$ and the hyperplane defined by $\sum_i^n x_i = \sum_i^n x_i(0)$. The open-loop command that would asymptotically lead to an equilibrium point $\e$ lying in the reachable set can be defined by taking any value of the parameter set from $\sigma_{\e}$.

\section{Conclusion}

We have studied the equilibrium locus of the flow generated by a circular network of an arbitrary number of cells. Such a model appears for example in the ribosome flow model on a ring. We have proven that for every value of the parameter vector $\ld$, the equilibrium locus is a smooth curve. Moreover, we have shown how to effectively compute the equilibrium point for a given value of $\ld$ and a given value of first integral made by the sum of system coordinates. 

We have proven that when considering the set of all possible values of the parameters, the equilibrium locus is a smooth manifold with boundaries, while for a given value of the parameters, it is an embedded smooth and connected curve. For different values of the parameters, the curves are all isomorphic. 

Moreover, we have shown how to build a homotopy between different curves obtained for different values of the parameter set. This procedure allows the efficient computation of the equilibrium point for each value of some first integral of the system. This point would have been otherwise difficult to be computed for higher dimensions. We have also performed some numerical experiments that illustrate this construction.   

Finally, we considered the control of such a system in open-loop. We showed how the system can be driven asymptotically to any reachable equilibrium, when the parameter set can be viewed as a vectorial input. 

\section*{Acknowledgement}
I am grateful to Prof. Michael Margaliot for having introduced me to the subject and for useful discussions. I would like also to thank Dr. Serge Lukasiewicz for useful remarks.

\bibliographystyle{amsplain}

\end{document}